\documentclass{amsart}

\newtheorem{theorem}{Theorem}
\newtheorem{proposition}[theorem]{Proposition}

\newtheorem{corollary}[theorem]{Corollary}
\newtheorem{question}[theorem]{Question}
\newtheorem{problem}[theorem]{Problem}

\numberwithin{equation}{section}
\numberwithin{theorem}{section}

\newcommand{\mb}{\mathbb}
\newcommand{\wh}{\widehat}

\newcommand{\R}{\mathbb{R}}

\newtheorem{namedtheorem}[theorem]{\theoremname}
\newcommand{\theoremname}{testing}
\newenvironment{named}[1]{\renewcommand{\theoremname}{#1}
        \begin{namedtheorem}}
        {\end{namedtheorem}}

\usepackage{amsfonts, amsmath, amsthm, amssymb, epsfig, bm, graphics,
  psfrag, latexsym, chngcntr, color, yfonts, mathtools, subfig, longtable,todonotes,enumitem,tikz-cd,appendix,enumitem,hyperref}
\usepackage[mathscr]{eucal}

\begin{document}
\title{Knot Floer homology of Positive Braids}
\author{Zhechi Cheng}
\address{School of Mathematics and Statistics, Wuhan University, Wuhan, Hubei, China}
\email{\href{mailto:zcheng@whu.edu.cn}{zcheng@whu.edu.cn}}
\author{Matthew Hedden} \address{Department of Mathematics, Michigan
  State University, East Lansing, MI}
\email{\href{mailto:matthew.hedden@gmail.com}{matthew.hedden@gmail.com}}

\begin{abstract} We compute the next-to-top term of the knot Floer homology of any link obtained as the closure of a positive braid, showing in particular that the rank is one for any prime  knot  in this family. As such knots are fibered, it follows that their monodromies are fixed-point free.  We compare the set of positive braids with other classes of knots known to have this property. One such class consists of knots possessing ``diagonal" grid diagrams. We provide an example of such a knot that is not a positive braid, providing an answer to a question of Vance and Kubota. We conclude with a number of problems and questions for future study naturally motivated by our theorem.

\end{abstract}
\maketitle 

\setcounter{tocdepth}{1}
\tableofcontents

\section{Introduction}
Knot Floer homology was introduced  in \cite{OSknot,Rthesis} as a bigraded refinement, for knots and links, of Ozsv\'ath and Szab\'o's singly graded Heegaard Floer homology groups of a homology 3-sphere.  The additional ``Alexander" grading  contains important information about $K$ not seen by the Alexander polynomial, knot Floer homology's graded Euler characteristic. For example, while the  Alexander polynomial bounds the Seifert genus and obstructs fiberedness, knot Floer homology determines them \cite{OSgenusbound,GhigginiGenus1,Ni06}. In these and most other earlier results about the geometric meaning of knot Floer homology, attention was focused solely on the  \textit{top term} $\wh{HFK}(K,g(K))$; that is, the non-trivial subgroup of knot Floer homology with highest Alexander grading. 

Recent years have given rise to results suggesting similarly rich information contained in the  \textit{next-to-top term} $\wh{HFK}(K,g(K)-1)$.  Thus far, these pertain exclusively to the class of fibered knots, highlighted by work of Ni \cite{Nifp2} which shows that if the rank of the next-to-top term of a fibered knot is $r$, then its monodromy is freely isotopic to a diffeomorpism with at most $(r-1)$ fixed-points, and  of Ghiggini-Spano \cite{GS22}  proving that this is typically  minimal among area-preserving diffeomorphisms  in its isotopy class. One should view this as a refinement of the fact that the second highest coefficient of the Alexander polynomial is one less than the Lefschetz number of the monodromy.

It follows that fibered knots with rank one next-to-top term have fixed point free monodromies, and it is natural to ask which knots have this property. In this paper, we establish that every prime knot arising as the closure of a braid word written entirely in the positive Artin generators is among them.  We refer to such knots as \textit{positive braid} knots.  A foundational result of Stallings \cite{Stallings} shows that (non-split) positive braid closures are fibered. 

In fact, we compute the next-to-top term of the knot Floer homology of any positive braid {\em link} $L$. The answer is stated in terms of the following quantities:
\vskip5pt
$\ |L|\ :=\text{the number of components of } L$, 
\vskip3pt
$g(L):=\underset{\Sigma}{\text{min}}\{\frac{|L|-\chi(\Sigma)}{2} |\ \Sigma \text{ is a (possibly disconnected) Seifert surface for } L\}$,rtyu
\vskip3pt
$s(L):=\text{ the number of split factors of }L$,
\vskip4pt
$p(L):=\text{ the number of prime factors for }L$. 
\vskip5pt
\noindent We refer to $g(L)$ defined above as the {\em genus} of $L$, and clarify that the number of prime factors of $L$ is the sum, over the split components $L_i$, of the number of non-trivial connected summands of $L_i$. In particular $p(\text{unlink})=0$.

\begin{theorem}\label{main} For $L$  a positive braid link, the next-to-top term of knot Floer homology satisfies $\wh{HFK}(L,g-1)\cong\mathbb{F}^{p(L)+|L|-s(L)}[-1] \otimes (\mathbb{F}[0]\oplus\mathbb{F}[-1])^{\otimes s(L)-1}$. 
\end{theorem}
\noindent In the above, we use $[i]$ to denote the Maslov grading shift of $i$, and remark  that there are two natural  conventions for the Maslov grading of the knot Floer homology of a link; we use the one taking integral values from \cite{OSlink} rather than the one that takes half-integral values found in \cite{OSknot}.  
The theorem can be viewed as a categorification of Ito's calculation of the penultimate coefficient of the Alexander polynomial of a positive braid closure \cite[Theorem 2]{Ito22}. 

Combined with Ni's result, we obtain the desired conclusion about fixed points: 

\begin{corollary}\label{cor:fp} Let $K$ be a prime positive braid knot with fiber surface $F$ and monodromy $\varphi$. Then $\varphi$ is freely isotopic to a diffeomorphism with no fixed point.
\end{corollary}

Theorem \ref{main} also yields a new proof of Krishna's result:
\begin{corollary}\cite[Theorem 1.5]{Krishna}\label{cor:cable} When $n\ge 2$, the $(n,1)$-cable of a knot $K$ is a positive braid knot if and only if $K$ is the unknot. 
\end{corollary}
\noindent It would be interesting to use our result in conjunction with cabling formulae \cite{HeddenCabling,LOT, HomThesis,HanselmanWatsonCabling,chen2024lspacesatelliteoperatorsknot} for knot Floer homology to determine when the $(n,m)$ cable of a knot is a positive braid as a function of $m$ and $K$. 

Indeed, as fibered ``strongly-quasipositive" braid closures are characterized by  \cite{Ni07,H10} in terms of the  top group of knot Floer homology and the canceling differential acting upon it (see \ref{prop:cancel} below), one might hope that our additional obstruction to a link being a (genuinely) positive braid closure provides a complete characterization of this class. While this holds through 12  crossing knots, the 13 crossing knot $K_{13n_{4639}}$ is a fibered strongly-quasipositive knot with next-to-top term consistent with Theorem \ref{main}, yet is not a positive braid closure \cite{BLL18}. One may further speculate whether our result picks out the positive braids from among the fibered positive knots. This, too, turns out to be false; using SnapPy \cite{snappy} we found the lowest crossing fibered positive, but not braid positive, knot whose next-to-top Floer homology is $\mathbb{F}[-1]$, is the 14 crossing  $K_{14n_{5644}}$.

Lastly, it bears mentioning that there are other classes of fibered knots whose next-to-top term is of rank $1$; for instance, \textit{L-space knots} \cite[Corollary 9]{HW18}, which are knots with a surgery to a manifold with simplest possible Heegaard Floer homology, and prime \textit{diagonal knots}, which are knots represented by grid diagrams whose $O$-markings all lie on the diagonal \cite[Corollary 1.9]{kubota25}. While neither of these latter two classes contain the other (see \cite[Theorem 1.17]{kubota25}),  Vance and Kubota asked whether the set of positive braid knots coincides with the set of diagonal knots \cite[Question 1.4]{kubota25}. The aforementioned 14 crossing knot provides a negative answer: 

\begin{proposition}
\label{thm:diagonal} There are diagonal knots that are not positive braids. 
\end{proposition}

 \noindent Figure \ref{fig:grid} depicts a diagonal grid diagram for $K_{14n_{5644}}$. Note that the knot $K_{13n_{4639}}$ is an L-space knot but not braid positive and the knot $K_{10_{139}}$ is a positive braid knot but not an L-space knot \cite{BLL18}, \cite{Knotinfo}. We suspect that there are examples of positive braids which are not diagonal, and further that all eight regions of the Venn diagram for these three classes of knots are populated.

\subsection*{Organization} In the next section we review positive braids, and their features used in the proof of the main theorem. The latter is proved in Section \ref{sec3}, which begins with a discussion of algebraic structures present in knot Floer homology that are essential to our method of attack, chief among them being the skein exact triangle and its interaction with a canceling differential acting on the knot Floer homology groups.  In fact we prove a stronger theorem than what is stated above, which determines the action of the canceling differential on the next-to-top term. The final section highlights a number of questions and future directions which our work naturally motivates.

\subsection*{Acknowledgement} The authors thank Dongtai He, Tetsuya Ito, Hajime Kubota, Robert Lipshitz for many helpful discussions and comments. Zhechi Cheng was partially supported by NSFC grant No. 12401087.  Matt Hedden gratefully acknowledges support from NSF DMS-2104664,  a Simons Fellowship, and  Cambridge University for a hospitable environment provided during Spring 2025, in which some of this work was done.

\section{Positive Braids}\label{sec2}
In this section, we review some basic properties of positive braid links. Let $B_n$ be the Artin braid group on $n$ strands, with generators $\sigma_1,\ldots,\sigma_{n-1}$. A braid is called {\em positive} if it can be expressed by a word using only the positive generators $\sigma_1,\ldots,\sigma_{n-1}$ and not their inverses $\sigma^{-1}_1,\ldots,\sigma^{-1}_{n-1}$.   By joining the ends of each strand, a braid $\beta\in B_n$ can be converted to a link $L_\beta$, referred to as {\em the closure} of  $\beta$. A link will be called a {\em positive braid link} if it is isotopic to the closure of a positive braid.   

The natural diagram associated to a positive braid closure has a number of useful features.  For instance, one can use it to easily determine if the underlying link type is split or composite.  To understand this, recall that a link diagram $D\subset \R^2$ is {\em split} if there is circle embedded in the complement of $D$ which separates it into multiple components.  Similarly, we say a diagram $D$ is {\em  composite} if  there is a circle intersecting $D$ in two points and such that the piece of $D$ on either side of the circle is not a trivial arc i.e. an arc without double points.  We have the following theorem of Cromwell:

\begin{theorem}\cite{Cromwell}\label{split} Let $L$ be a positive braid link and $\hat\beta$ denote the natural diagram of any closed positive braid representing $L$. Then $L$ is split  if and only if the link diagram $\hat\beta$ is split.  Similarly, if $\hat\beta$ has no nugatory crossings (a double point of the diagram which separates it) and is not split, then $L$ is composite if and only if $\beta$ factors as $\beta=\beta_1\circ \beta_2$ for positive braids $\beta_1$ and $\beta_2$ on the first $k$ and last $n-k$ strands, respectively. In particular, the diagram of $\hat\beta$ is composite with summands $\hat\beta_1$ and $\hat\beta_2$.
\end{theorem}

\noindent Of course a link is split or composite if it possesses a diagram which is--the content here is that {\em any} positive braid diagram will have these features.  For this reason positive braid links are said to be \textit{visually split} and \textit{visually prime} respectively, a feature which they share with several other classes of diagrammatically defined links.  For instance, the above theorem holds for alternating links by work of Menasco \cite{Menasco}, for positive links by Ozawa  \cite{Ozawa}, and for homogeneous braid links by Feller and Lewark  \cite{FLR}. 

 Theorem \ref{split} allows us to assume that a prime, non-split, positive braid link is represented by a braid in which every generator appears at least twice. Indeed, if $\sigma_i$ does not appear, then the diagram splits. If $\sigma_i$ appears only once then the diagram is composite, decomposed into the connected sum of two diagrams. If one of these diagrams represents an unknot, so that the closure is still prime, we can simply remove $\sigma_i$ and the subword of the braid corresponding to the unknot.  This results in a positive braid whose closure has the same isotopy class, but with strictly smaller braid index.  Repeating this argument as necessary, one eventually arrives at a positive braid where each generator appears at least twice. Perhaps surprisingly, we can moreover assume that there is a generator appearing twice {\em in succession}:

\begin{proposition}\label{prop:adjacentgenerator} Let $L$ be a positive braid link which is not the unlink. Then there is positive braid representative of $L$ of minimal braid index of the form $\sigma_i^2\beta$ for some $i$ and word $\beta$. 
\end{proposition}
\proof This can be proved by a straightforward inductive argument using the braid relation, and appears in numerous places; for instance \cite[Lemma 2]{Buskirk} \cite[Lemma 2.1]{FranksWilliams} \cite[Claim 4.4]{Nakamura}  \cite[Lemma 1]{BaaderDehornoy} \cite[Theorem 8]{Ito22}.
\qed

Our analysis of positive braids will hinge on inductive arguments involving links related by the \textit{oriented skein relation}.  To describe this, and to establish notation, let $L_+$ be a link represented by a diagram that contains a distinguished positive crossing, $L_-$ be the link defined by the  diagram obtained by changing this crossing, and $L_0$ be the link defined from the diagram gotten by  resolving the crossing in the unique  orientation-respecting manner, see Figure \ref{fig:skein}.
\begin{figure}
\centering
\includegraphics[width=0.6\linewidth]{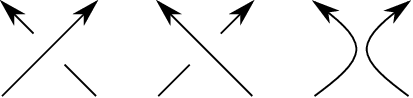}
\caption{The skein relation: from left to right, we have $L_+$, $L_-$ and $L_0$.}
\label{fig:skein}
\end{figure} In these terms, the utility of Proposition \ref{prop:adjacentgenerator} becomes clear. Indeed, if $L_+$ is the closure of a positive braid $\sigma_i^2\beta$, and we consider a skein relation where one of the two adjacent crossings coming from $\sigma_i^2$ is distinguished (see Figure \ref{fig:adjacent}), then $L_0$ will be the closure of $\sigma_i\beta$ and $L_-$ will be isotopic to the closure of $\beta$. All three of these braids are positive.  This allows us to reduce the length of the braid using skein relations,  whilst staying strictly within the realm of braid braids. 

\begin{figure}
\centering
\includegraphics[width=0.6\linewidth]{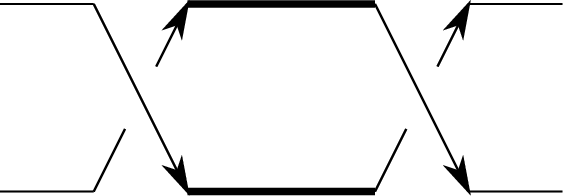}
\caption{The two crossings are adjacent if there is no other crossing in the thickened segments.}
\label{fig:adjacent}
\end{figure}

It will be useful to note that in the situation considered above, the genera of the three positive braid closures satisfy:
\begin{equation}\label{eqt:genus} g(L_+)=g(L_-)+1=g(L_0)+\delta,
\end{equation}
where $\delta$ is $0$ when $L_0$ has more components than $L_\pm$, and $1$ otherwise. Indeed, this relation amongst the genera of $L_+,L_-$ and $L_0$ holds for adjacent positive crossings in any positive  diagram of $L_+$, as the genus of a knot is realized by the Seifert surface provided by Seifert's algorithm applied to any positive  diagram \cite[Corollary 4.1]{MR1002465}.

\section{Knot Floer Homology of Positive Braids}\label{sec3}
We start by reviewing some key properties of knot Floer homology. For knots $K$ in $S^3$, we associate a Maslov-graded chain complex $\wh{CFK}(K)$ over $\mb{F}=\mb{Z}/2\mb{Z}$, filtered by an additional Alexander-grading. The associated-grade homology with respect to the latter is called the knot Floer homology $\wh{HFK}(K)$, defined in \cite{OSknot} and \cite{Rthesis}, where it is shown to be a knot invariant that categorifies the Alexander polynomial. 

There are a several ways to generalize this theory to links. For the present purposes, we follow \cite{OSknot} and consider the \textit{knotification} of a link $L$ in $S^3$. This is a well-defined knot $\kappa(L)$ in the connected sum of $(|L|-1)$ copies of $S^1\times S^2$, obtained by joining the components of $L$ by bands running through the essential 2-spheres in the summands. As knot Floer homology is defined for null-homologous knots $K$ in any closed $3$-manifold $Y$,  one can therefore define $\wh{HFK}(S^3,L)$ to be $\wh{HFK}(\#S^1\times S^2, \kappa(L))$. 

Key to our computation is an exact triangle relating the knot Floer homology of links in Figure \ref{fig:skein}:

\begin{named}{Skein Exact Triangle}\cite{OSknot,OSes}\label{prop:es} If $L_+$ has one fewer component than $L_0$, then there is an exact sequence:
\begin{equation} \ldots\to\wh{HFK}_m(L_+,a)\overset{f_a}\to \wh{HFK}_m(L_-,a)\overset{g_a}\to \wh{HFK}_{m-1}(L_0,a)\overset{h_a}\to\dots.
\end{equation}
 If $L_+$ has one more component than $L_0$, we have
\begin{equation} \ldots \to\wh{HFK}(L_+,a)\to \wh{HFK}(L_-,a)\to \wh{HFK}_{m-1}(L_0\#B,a)\to\ldots,
\end{equation}
 where $B$ is the genus one fibered knot with identity monodromy.  Equivalently, $B\subset \#^2S^1\times S^2$ is the knot obtained by doing zero surgery on two-components of the Borromean rings.
 \end{named}

 To understand the maps in the exact sequences, we will exploit an additional structure on knot Floer homology: 

\begin{named}{Canceling Differential}\cite{OSknot} \label{prop:cancel} There is an endomorphism $\partial_L$ of the knot Floer homology group $\oplus_{m,a}\wh{HFK}_m(L,a)$ whose restriction to each summand is \[\partial_L|_{\wh{HFK}_m(L,a)}: \wh{HFK}_m(L,a) \rightarrow \underset{a'<a}\oplus\wh{HFK}_{m-1}(L,a'),\]
and which satisfies $\partial_L\circ \partial_L = 0$.  The homology of the resulting chain complex is:
\[ H_*(\wh{HFK}(L),\partial_L) \cong (\mathbb{F}[0]\oplus\mathbb{F}[-1])^{\otimes |L|-1}.\]
\end{named}

One should view the homology as the (shifted) Floer homology of  the connected sum of $S^1\times S^2$'s in which the knotification of $L$ is embedded.  The canceling differential arises from the definition of the knot Floer homology groups as the homology of the associated graded complex of the Alexander filtration.  While one could alternatively work with the spectral sequence associated to this filtration, the canceling differential perspective is often simpler.  In particular, the maps in the skein exact triangle can be viewed as the Alexander grading-preserving components of (filtered) chain maps, with respect to the canceling differential. More precisely, for each $m$ there are maps:

\[ f: \underset{a}\oplus\wh{HFK}_m(L_+,a) \rightarrow \underset{a}\oplus\wh{HFK}_{m}(L_-,a)\]
which satisfy $f \partial_{L_+}=\partial_{L_{-}} f$. This chain map is filtered, in the sense that it respects the Alexander grading: \[\mathrm{Im}(f|_{\wh{HFK}_m(L_+,a)})\subset \underset{a'\le a}\oplus\wh{HFK}_{m}(L_-,a').\] The components of $f$ which preserve the Alexander-grading are exactly the maps $f_a$ appearing in the skein exact sequence.   There are, likewise, chain maps $g$ and $h$ refining the maps along the other two edges in the exact triangle.   

Much of the utility in this structure lies in the fact that the  maps induced on homology, which we denote $F,G,H$, are simply those appearing in the surgery exact triangle relating the (shifted) Floer homology of $\infty,-1$ and $0$-surgeries along an unknot in the connected sum of $|L_+|-1$ copies of $S^1\times S^2$. In particular, if we let $V= (\mathbb{F}[0]\oplus\mathbb{F}[-1])^{\otimes{|L_+|-1}} $, then the maps on homology with respect to the canceling differential split as:
\begin{equation}\label{eq:surgery} 0 \to V\overset{G}\hookrightarrow   V[-1]\oplus V[0] \overset{H}\twoheadrightarrow  V\overset{F}\to 0,
\end{equation}
where $G$ maps $V$ into the summand $V[-1]$ via an isomorphism shifting the Maslov degree down by $1$,  and $H$ maps $V[0]$ to $V$ via a grading preserving isomorphism.  We note that while the statement of the skein sequence depends on whether $|L_+|=|L_-|$ is greater or less than $|L_0|$, the groups and maps in  \eqref{eq:surgery} do not.  This is a consequence of the fact that the ambient manifold of the genus one fibered knot with identity monodromy is $\#^2S^1\times S^2$.
In fact, the knot Floer homology of $L_0\#B$ is simply the tensor product of $\wh{HFK}(L_0)$ with the bigraded vector space $J=\mathbb{F}[0,1]\oplus \mathbb{F}^2[-1,0]\oplus \mathbb{F}[-2,-1]$, endowed with the trivial (identically zero) canceling differential. This is a consequence of a general K{\"u}nneth type theorem for the Floer homology of connected sums and disjoint unions:

\begin{named}{Product Formulae}\cite[Theorem 7.1]{OSknot}\label{PF}
\[ \wh{HFK}(Y_1\#Y_2,L_1\#L_2)\cong \wh{HFK}(Y_1,L_1)\otimes\wh{HFK}(Y_2,L_2)\]
\[\wh{HFK}(Y_1\#Y_2,L_1\sqcup L_2)\cong \wh{HFK}(Y_1,L_1)\otimes\wh{HFK}(Y_2,L_2)\otimes (\mathbb{F}[0]\oplus\mathbb{F}[-1]). \]

\end{named}
\noindent We remark that while the connected sum of links depends on choices of components, the above formula holds for {\em any} such choices. We further note that the tensor products are taken in the bigraded sense, and equate the canceling differentials on the left-hand side with the Leibniz product of those on the right:  
\begin{equation}\begin{split} \label{Leibniz}
\partial_{L_1}\otimes 1_{L_2} &+ 1_{L_1}\otimes \partial_{L_2},\\
\partial_{L_1}\otimes 1_{L_2}\otimes 1 &+ 1_{L_1}\otimes \partial_{L_2}\otimes 1.
\end{split}
\end{equation} Note that the differential for the split union case given in the second row is indeed the Leibniz product, when we endow the vector space $\mathbb{F}[0]\oplus\mathbb{F}[-1]$  with the zero differential.

As a final preparatory lemma, we compute the top group of the knot Floer homology of a positive braid link:
\begin{proposition}\cite{Stallings,OScontact}\label{top} If $L$ is a positive braid link, then $\wh{HFK}(L,g)=(\mathbb{F}[0]\oplus\mathbb{F}[-1])^{\otimes{s(L)-1}}$. The canceling differential $\partial_L$ acts on the top group trivially, i.e. every element is a cycle.
\end{proposition}
\proof We first address the non-split case, when $s(L)=1$, noting that by convention $(\mathbb{F}[0]\oplus\mathbb{F}[-1])^{\otimes{0}}=\mathbb{F}[0]$.  A non-split positive braid link is fibered \cite[Theorem 2]{Stallings}, which implies $\wh{HFK}(L,g)\cong \mathbb{F}$ \cite[Theorem 1.1]{OScontact}.  Since the knot Floer homology vanishes in Alexander gradings greater than $g$, and the canceling differential strictly lowers it, a generator of this group is a Maslov graded zero cycle if and only if its homology class generates the Floer homology of $\#^{|L|-1}S^1\times S^2$ in highest Maslov grading.  This, in turn is equivalent to $\tau_{top}(L)=g$, which was proved for fibered positive braid links in \cite[Theorem 5.22]{HR22} (cf. \cite{Livingston,Cavallo}).    The general case follows from the product formula for split links, noting that each of the $s(L)$ split components will be a non-split positive braid link. 
\qed

\vspace{0.07in}
\noindent With the set-up  in place, we turn to the proof of Theorem \ref{main}. We actually prove the following stronger result:

\begin{theorem}\label{mainprime} For $L$  a positive braid link, the next-to-top term of knot Floer homology satisfies { \[ \wh{HFK}(L,g-1)\cong\mathbb{F}^{p(L)+|L|-s(L)}[-1] \otimes (\mathbb{F}[0]\oplus\mathbb{F}[-1])^{\otimes s(L)-1}.\]} 

\noindent With respect to the above tensor product, the canceling differential $\partial_L$ restricted to this group splits as $\partial\otimes 1$, where  $\partial$ has rank $p(L)$. 
\end{theorem}
\noindent In particular, the rank of the restriction of  $\partial_L$ to the next-to-top term has rank $2^{s(L)-1}\cdot p(L)$.
\proof Given a positive braid link $L$, consider any diagram for $L$ arising from the closure of a positive braid representative.  We induct on the number of crossings in such a diagram, which we denote by $w$. The base case, where $w=0$, corresponds to an unlink $U$ of $n$ components, where $n$ is the braid index. We have $g(U)=0$, and the knot Floer homology $\wh{HFK}(U,-1)$ vanishes, by direct calculation (or the fact that knot Floer homology detects genus). This agrees with the assertion of Theorem \ref{mainprime}, since the exponent $p(L)+|L|-s(L)$ of the $\mathbb{F}[-1]$ summand is zero ($=0+n-n$).

Assume, then, that the theorem has been proved for positive braid closures with fewer than $w$ crossings, and consider a diagram for $L$ with $w$.  We may assume none of these crossings are nugatory, since they could be removed with a half twist that results in a positive braid of fewer crossings. We first reduce to the case that $L$ is both prime and non-split.  Both reductions will follow from the inductive hypothesis, together with the Product Formulae \ref{PF}.
\vspace{0.05in}

\noindent{\bf Reduction to prime, non-split links:} We begin by treating the case of connected sums of non-split links; that is,  $L=L_1\#L_2$ where  $s(L)=s(L_i)=1$. Cromwell's Theorem \ref{split} implies the braid itself is reducible, factoring as a product $\beta_1(\sigma_1,\ldots, \sigma_{k-1})\beta_2(\sigma_k,\ldots,\sigma_{n-1})$ of two positive braids on the first $k$ and last $n-k$ strands.  The inductive hypothesis implies that Theorem \ref{mainprime} holds for each of $L_1=L_{\beta_1}$ and $L_2=L_{\beta_2}$.
Applying the product formula \ref{PF} for connected sums and Proposition \ref{top}, we compute the next-to-top term of $L_1\# L_2$:
 \begin{align*}
    \wh{HFK}(L_1 &\# L_2,g_1+g_2-1)\\&\cong (\wh{HFK}(L_1,g_1)\otimes\wh{HFK}(L_2,g_2-1))\oplus(\wh{HFK}(L_1,g_1-1)\otimes\wh{HFK}(L_2,g_2))\\
    &\cong \mathbb{F}^{p(L_2)+|L_2|-1}[-1]\oplus \mathbb{F}^{p(L_1)+|L_1|-1}[-1].
\end{align*}
Noting that $p(L_1\#L_2)=p(L_1)+p(L_2)$ and $|L_1\#L_2|=|L_1|+|L_2|-1$, we see that the next-to-top term of $L_1\#L_2$ satisfies Theorem \ref{mainprime} as well. The Leibniz formula for the canceling differential \ref{Leibniz} implies that the rank of its restriction to the next-to-top term is additive under connected sums.  Hence, by induction, the differential has the form asserted in the theorem. 

We now turn to the split case, with $s(L)>1$. Cromwell's result implies that the positive diagram considered is split by circles in the plane into the disjoint union of positive braid diagrams for non-split links $L_i$, $i=1,\ldots, s(L)$. We assume that all $L_i$ have crossings, because the product formula for split links implies that the validity of Theorem \ref{mainprime} is preserved upon taking the disjoint union of an unlink with any link satisfying its conclusion. The product formula also identifies the next-to-top term of $L_1\sqcup\ldots\sqcup L_{s(L)}$ with the tensor product of  the next-to-top term of a non-split link $L_1\#\ldots\#L_{s(L)}$ and $({\mathbb{F}[0]\oplus\mathbb{F}[-1]})^{s(L)-1}$. Moreover, under this identification the canceling differential is given by $\partial_{L_1\#\ldots\#L_{s(L)}}\otimes 1$,  where $1$ denotes the identity map on the $({\mathbb{F}[0]\oplus\mathbb{F}[-1]})^{s(L)-1}$ factor. On the other hand, we can see that the exponents of the $\mathbb{F}[-1]$ factors in Theorem \ref{mainprime} agree for the connected sum and disjoint union: 
\begin{align*}
    p(\#L_i)&+|\#L_i|-s(\#L_i)\\&=\sum_{i=1}^{s(L)}p(L_i)+\big((\sum_{i=1}^{s(L)}|L_i|)-(s(L)-1)\big)-1\\&=p(L)+|L|-s(L).
\end{align*}
Since we showed that the next-to-top term of the non-split link $L_1\#\ldots\#L_{s(L)}$ satisfies Theorem \ref{mainprime}, the above calculation shows that it holds for $L_1\sqcup\ldots\sqcup L_{s(L)}$ as well. 

 We will therefore assume that $L$ is a non-trivial link which is non-split and prime; that is, $s(L)=1$ and $p(L)=1$.   We wish to prove that $\wh{HFK}(L,g-1)\cong\mathbb{F}^{|L|}[-1]$, and that $\partial_L$ has rank one when restricted to this group.  According to Proposition \ref{prop:adjacentgenerator}, we can assume that the positive braid representing $L$ is of the form $\sigma_i^2\beta$.  We distinguish one of the adjacent positive crossings coming from $\sigma_i^2$, and consider the resulting skein exact sequence with $L_+=L$, $L_-$ isotopic to the closure of $\beta$, and $L_0$ being the closure of $\sigma_i\beta$. There are three cases to consider: $(1)$ $\sigma_i$ appears more than once in $\beta$; $(2)$ $\sigma_i$ appears once in $\beta$; $(3)$ $\sigma_i$ does not appear in $\beta$. 
\vspace{.07in}

\noindent{\bf Case 1:} Here, both $L_-$ and $L_0$ will be non-split and prime, by Theorem \ref{split}.  We consider the skein exact triangle, restricted to Alexander grading $a=g(L_+)-1=g(L_-)$, so from Proposition  \ref{top} we know $\wh{HFK}(K_-,a)\cong \mb{F}[0]$. There are two subcases, depending on whether $|L_0|>|L_+|$ or not. If $|L_0|>|L_+|$, Equation \ref{eqt:genus} shows that $a=g(L_0)-1$. The skein exact triangle \ref{prop:es} in this Alexander grading yields: 
\begin{align}\begin{split}\label{eq:es} \ldots\to\wh{HFK}_1(L_+,a)&\overset{f_a}\to 0\overset{g_a}\to \wh{HFK}_{0}(L_0,a)\overset{h_a}\to \\
\to\wh{HFK}_0(L_+,a)&\overset{f_a}\to \mb{F}\overset{g_a}\to \wh{HFK}_{(-1)}(L_0,a)\overset{h_a}\to\\
\to\wh{HFK}_{(-1)}(L_+,a)&\overset{f_a}\to 0\overset{g_a}\to \wh{HFK}_{(-2)}(L_0,a)\overset{h_a}\to\ldots.
\end{split}
\end{align}
\normalsize All the other terms in the middle column are $0$, which implies $\wh{HFK}_m(L_+,a)\cong\wh{HFK}_m(L_0,a)$ for $m\ne0,-1$, and all of these groups are $0$ by our inductive hypothesis.  We claim that the $\mathbb{F}$ in the middle column injects into $\wh{HFK}_{(-1)}(L_0,a)$ under  $g_a$.  Indeed, by Proposition \ref{top} the $\mathbb{F}$ generates the (rank one) subgroup of $H_*(\wh{HFK}(L_-),\partial_{L_-})$ in (highest) Maslov grading zero.  On the other hand, the inductive hypothesis shows  that $\wh{HFK}_{(-1)}(L_0,a)\cong \mb{F}^{|L_0|}$ and that ker$(\partial_{L_0})$ is an ${|L_0|}-1$ dimensional subgroup, which therefore generates  $H_{(-1)}(\wh{HFK}(L_0),\partial_{L_0})$.   Now the map $g_a$ is the highest order component, with respect to the Alexander grading, of a chain map $g$ whose map on homology is the $G$ from Equation \ref{eq:surgery}.  This latter map is injective, and hence maps the $\mathbb{F}$ in grading zero into $H_{(-1)}(\wh{HFK}(L_0),\partial_{L_0})$.  Since this latter subgroup is supported entirely in Alexander grading $a$, it follows that $g_a$ must be injective on $\mathbb{F}[0]$ as well.

Having established injectivity of $g_a$ exactness of Equation \ref{eq:es} shows that $f_a$ is zero, $h_a$ is surjective, and hence $\wh{HFK}_{(-1)}(L_+,a)\cong \wh{HFK}_{(-1)}(L_0,a)/\mathrm{Im}(g_a)\cong \mb{F}^{|L_+|}$, thus establishing the first part of the theorem for $L_+$. 

We turn to the claim about the canceling differential.   Let $\{x,y_1,...,y_n,z\}$ be a basis for $\wh{HFK}_{(-1)}(L_0,a)$ such that $x=\mathrm{Im}(g_a)$, $\partial_{L_0}(y_i)=0$, and $
\partial_{L_0}(z)\ne 0$.  Exactness of \ref{eq:es} implies $h_a$ maps the subgroup generated by $\{y_1,...,y_n,z\}$ isomorphically onto $\wh{HFK}_{(-1)}(L_+,a)$.  Letting $h$ denote the filtered chain map discussed above, we have $\partial_{L_+}(h(y_i))=h(\partial_{L_0}(y_i))=0$.  It follows that $h(y_i)=h_a(y_i)+h_{<a}(y_i)$ is a cycle with respect to the canceling differential, where $h_{<a}$ denotes the components of $h$ that strictly lower the Alexander grading.  Thus, up to a filtered change of basis, we can assume that $h_a$ maps the subspace spanned by $\{y_1,...,y_n\}$ into a subspace of $\wh{HFK}_{(-1)}(L_+,a)$ in the kernel of $\partial_{L_+}$.  Since the dimension of this subspace is $|L_+|-1$, which agrees with the dimension of $H_{(-1)}(\wh{HKF}(L_+),\partial_{L_+})$, $h_a(z)$ cannot be in the kernel of $\partial_{L_+}$, and the rank of $\partial_{L_+}$ is therefore one, as claimed.  

The subcase when $|L_0|<|L_+|$ is algebraically identical. Indeed, by Equation \ref{eqt:genus}, the Product Formula \ref{PF}, and  the calculation of the knot Floer homology of $B\subset \#^2 S^1\times S^2$, the knot Floer homology of $L_0\#B$ and its canceling differential are isomorphic to that of a link with $|L_+|+1$ components which satisfies the theorem.

\vspace{.07in}

\noindent{\bf Case 2:} When $\sigma_i$ appears once in $\beta$, both $L_-$ and $L_0$ are non-split, and $L_0$ is prime. Letting $a=g(L_+)-1=g(L_-)$ as before, Proposition \ref{top} shows that $\wh{HFK}(L_-,a)$ is again $\mb{F}[0]$, and the inductive hypothesis  yields the same structure for $\wh{HFK}(L_0,a)$ and $\partial_{L_0}$ as in Case 1.  The argument with the skein exact triangle is therefore the same, reducing the proof to the final case.  
\vspace{.07in}

\noindent{\bf Case 3:} When $\sigma_i$ is not in $\beta$, $L_+$ is the connected sum of the positive Hopf link with links $L_1$ and $L_2$.  If either of these links are non-trivial, it contradicts our assumption that $L_+$ is prime, and hence we are reduced to the case of the Hopf link whose knot Floer homology is easily calculated and shown to satisfy Theorem \ref{mainprime} (see \cite[Figure 1]{HR22} for an explicit calculation).
\qed

\vspace{.07in}

\proof[Proof of Corollary \ref{cor:fp}:] By \cite[Theorem 1.2]{Ni23}, the monodromy $\varphi$ is freely isotopic to a diffeomorphism with at most $\mathrm{rank}\wh{HFK}(K,g-1)-1$ fixed points. In the case of prime positive braids, Theorem \ref{main} implies that $\mathrm{rank}\wh{HFK}(K,g-1)=1$ so the number of fixed points must be $0$. 
\qed

 \proof[Proof of Corollary \ref{cor:cable}:] The Alexander polynomial of the $(n,1)$ cable of $K$ is given by $\Delta_K(t^n)$, where $\Delta_K(t)$ is the polynomial of $K$ (see, e.g. \cite[Theorem 6.15]{Lickorish}). When $n\ge 2$,  the coefficient of the next-to-top term of $\Delta_K(t^n)$ is $0$. According to Theorem \ref{main}, if the $(n,1)$ cable were a positive braid closure this coefficient would be a non-zero negative integer unless $p(K_{n,1})=0$; that is, unless the cable of $K$, and hence $K$ itself, is unknotted.
\qed

\vspace{.07in}

We conclude this section by briefly straying outside the realm of positive braids, providing an example of a diagonal knot which is not among them.
\vspace{.05in}

\begin{figure}
\centering
\includegraphics[width=0.8\linewidth]{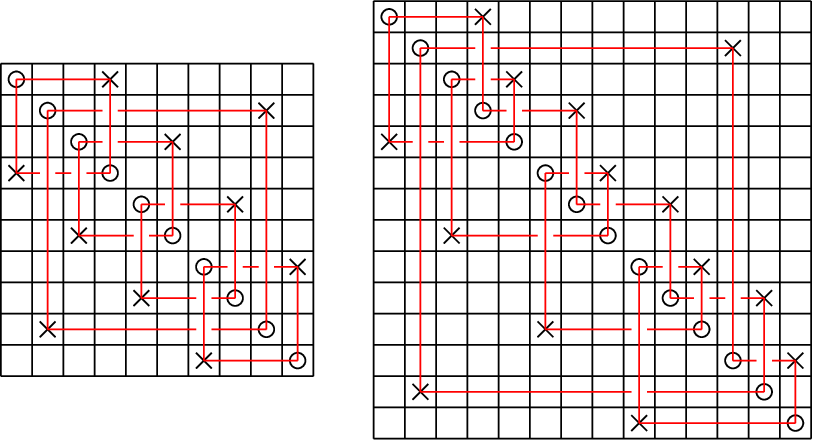}
\caption{The left diagram is a diagonal grid diagram for $R_5$, a chain link of $5$ unknots, from which we obtain $D_5=K_{14n_{5644}}$ on the right by inserting $4$ additional positive crossings. Adding more linked square unknotted components to the left diagram gives rise to $R_n$ for $n>5$, from which insertion of $n-1$ positive crossings produces $D_n$.}
\label{fig:grid}
\end{figure}

\proof[Proof of Proposition \ref{thm:diagonal}:] The diagram on the right in Figure \ref{fig:grid} is clearly diagonal.  By computing its grid homology, one identifies it with $K_{14n_{5644}}$.  Positive braid knots of genus less than six were enumerated in \cite[Section 7]{BLL18}, and the genus $5$ knot in question was not among them (it is, however, a plumbing of positive Hopf bands along a tree, which can be found in \cite[Figure 12]{BLL18}). 
\qed

We remark that the example above does not live in isolation. Indeed, we were led to it by considering a family of diagonal knots $D_n$,  which we expect not to be positive braid knots for $n\ge5$. The knot $D_n$ is obtained from a positive chain ring of  $n$ unknots, which we denote $R_n$, by inserting  $n-1$ positive crossings that connect the components together.   Figure \ref{fig:grid}  illustrates this process for  $n=5$. Considering a crossing in an adjacent pair, we have a skein exact sequence where $L_+=D_n$, $L_-=D_{n-1}$, and $L_0=D_{n-1}\#H$, where $H$ is the positive Hopf link.  In a similar vein, we can fit $R_n$ into a skein exact sequence with $R_{n-1}$ and a connected sum of $n-1$ positive Hopf links (a chain link).   Our inductive calculation of the next-to-top term of positive braids extends to these families,  from which we conclude that all links arising in these skein trees have next-to-top term isomorphic to $\mb{F}^{|L|}[-1]$.  Since the links in question are all diagonal, one should be able to compute this directly from the grid homology complex, as in \cite[Theorem 1.9]{kubota25}.

\section{Related Questions and Future Directions}\label{sec4}
Theorem \ref{main} sits at the nexus of surface dynamics, diagrammatic knot theory, Floer homology, and categorification, and motivates a number of questions in these avenues. 

For instance, while Corollary \ref{cor:fp} shows that prime positive braid closures have fixed-point free monodromies, the circuitous route through knot Floer homology and its relationship to symplectic Floer homology leaves opaque any geometric underpinning of this behavior. It is natural to ask:
\begin{question} Can one show that prime positive braid monodromies are fixed point free by direct examination?
\end{question}
\noindent Indeed, as their fiber surfaces are obtained by a straightforward plumbing of Hopf bands it is conceivable that one could offer a more intuitive explanation of their lack of fixed points.  Doing so might  lead to progress on other natural questions:
\begin{question} Do positive braid closures naturally fit within a broader class of fibered links with fixed point free monodromy? Where else do fixed point free monodromies arise ``in nature"?
\end{question}
\noindent A natural family, suggested by our  study of the knots $D_n$, are prime knots obtained by plumbing positive Hopf bands along a tree.  More generally, our inductive approach to the Floer homology of positive braids should extend Theorem \ref{main} to the family of links obtained by plumbing positive Hopf bands along a forest. 

One could also pursue extending our calculations to fibered links outside the realm of the 3-sphere.  For instance, the symplectic Floer homology of a natural class of closed surface automorphisms was considered by Eftekhary in \cite{Eftekhary} (cf. \cite{MR1426539,MR2039162}). It would be interesting to study the    analogous automorphisms of surfaces with boundary from the perspective of knot Floer homology.
\vskip0.05in

In this article we have only considered the {\em knot} Floer homology of a link, which categorifies the single variable Alexander polynomial.  The {\em link} Floer homology of a link categorifies the multi-variable Alexander polynomial, and determines the  classes in $H^1(Y\setminus L)$ which are realized by fibrations \cite{Altman}.   It would be interesting to extend results on fixed points to this situation:
\begin{problem} Bound the number of fixed points of {\em every} fibration of a link exterior using link Floer homology. 
\end{problem}
\noindent To this end, a relative homology class can be used to collapse the Alexander multi-grading on link Floer homology to a $\mathbb{Z}$-grading. It would be natural to expect the number of fixed points of the monodromy of a fibration of the link exterior to be bounded above by rank of the next-to-top group, computed with respect to the grading induced by the class of the fiber.  Assuming this is true, one could ask whether other fibrations on the complement of a positive braid link share similar dynamics to those demonstrated for Stalling's fibration by Theorem \ref{main}.

There are many well-studied notions of positivity for knots and links, and being a positive braid closure is in many regards the most stringent such notion one could employ.  For instance they are certainly {\em positive links}, in the sense that they admit diagrams with only positive crossings, and the latter are strictly contained in the family of {\em strongly quasipositive links} \cite{RudolphPSQP}.  One could speculate  what shadow of these features is cast on the knot Floer homology of fibered links, and the implications this might have for the dynamics of their monodromies.  While the next-to-top Floer homology of knots in these classes can be quite large, our calculations suggest these positivities may, like the positive braid condition, constrain the subgroup in Maslov grading $-1$.

\begin{question} For a prime fibered positive knot $K$, is \[\wh{HFK}_{(-1)}(K,g-1)\cong\mb{F}?\] If so, does this extend to prime strongly quasipositive fibered knots?
\end{question}

\noindent An affirmative answer for positive knots would imply that the number of prime factors of a fibered positive knot is equal to the rank of  $\wh{HFK}_{(-1)}(K,g-1)$ (this follows from the Product Formula \ref{PF} and the visual primeness of positive knots \cite{Ozawa}).  A further extension to strongly quasipositive fibered knots might suggest a connection between the Maslov grading $-1$ subgroup and contact geometry.  We note the above constraint is satisfied by prime fibered strongly quasipositive knots with up to 15 crossings.  

One could also attempt to employ our methods to study the knot Floer homology of positive braids beyond the top two Alexander gradings.  Kubota showed that the Floer homology of a diagonal knot in Alexander grading $g-2$ is supported in Maslov grading $-1$ unless it is a $(2,n)$ torus knot in \cite{kubota25}. It seems reasonable to ask whether the same holds for positive braids:
\begin{question}
For a prime positive braid knot $K$ that is not a $(2,n)$ torus knot, is \[\wh{HFK}(K,g-2)\cong\mb{F}^m[-1]?\] 
\end{question}

\noindent Pursuing such an understanding is motivated by the following
\begin{problem}
     Better understand the relationship between  $\wh{HFK}(K,g-2)$ and dynamics of the monodromy of $K$.
\end{problem}
   Indeed, by way of  Ozsv{\'a}th-Szab{\'o}'s surgery formula for knot Floer homology \cite{MR2377279}, work of Lee-Taubes \cite{LeeTaubes}, and Kutluhan-Lee-Taubes \cite{KLTI,KLTII,KLTIII,KLTIV,KLTV} cf. \cite{CGHI,CGHII,CGHIII}, this knot Floer homology group is closely related to the  periodic Floer homology group  generated by 2-periodic orbits of the capped-off monodromy defined by Hutchings \cite{MR1941088}.   
\bibliographystyle{alpha}
\bibliography{references}

\end{document}